\documentclass[halfparskip]{scrartcl}
\usepackage[latin1]{inputenc}
\usepackage{latexsym}
\usepackage{graphicx}
\usepackage{euscript}
\usepackage{amsfonts}
\usepackage{amssymb}
\usepackage{amsthm}
\usepackage{amsmath}
\usepackage[arrow, matrix, curve] {xy}
\usepackage{enumerate}

\newtheorem{theorem}{Theorem}

\newtheorem*{unntheorem}{Theorem}
\newtheorem{lemma}{Lemma}
\newtheorem{corollary}{Corollary}
\newtheorem{proposition}{Proposition}
\newtheorem{question}{Question}

\newcommand{\R}{\mathbb{R}}
\newcommand{\C}{\mathbb{C}}
\newcommand{\N}{\mathbb{N}}
\newcommand{\Z}{\mathbb{Z}}

\newcommand{\T}{\mathbb{T}}

\newcommand{\Exp}{\textnormal{Exp}}

\newcommand{\Bo}{{\cal B}}

\newcommand{\A}{{\cal A}}

\newcommand{\Imm}{\textnormal{Im}}
\newcommand{\Reee}{\textnormal{Re}}

\newcommand{\kernel}{\textnormal{ker}}
\newcommand{\ind}{\textnormal{Ind}}
\newcommand{\conv}{\textnormal{conv}}

\newcommand{\ldens}{\underline{\textnormal{dens}}}

\newcommand{\linspan}{\textnormal{span}}
\newcommand{\I}{I}

\title{\bf \center{On the intersection of the spectrum of frequently hypercyclic operators with the unit circle} }
\author{Hans-Peter Beise}

\begin{document}
\maketitle

%

\begin{abstract}
We exclude the existence of frequently hypercyclic operators that have a spectrum contained in the closed unit disc and that intersects the unit circle in only finitely many points under certain additional conditions. This extends a result of S. Shkarin, which states that the spectrum of a frequently hypercyclic operator cannot have isolated points.  
\end{abstract}


%
%


\section{Introduction and main results}
A bounded operator $T$ on a  Banach space $X$ is called \textbf{hypercyclic} if there exists a vector $x\in X$ such that the orbit $\{T^nx:n\in\N\}$ is dense in $X$. Such a vector $x$ is said to be a \textbf{hypercyclic vector}. The operator is called \textbf{frequently hypercyclic} if there exists some $x\in X$ such that for every non-empty open set $U\subset X$ the sequence $\{n:T^nx\in U\}$ has positive lower density. The vector $x$ is called a \textbf{frequently hypercylic vector} in this case. We recall that the \textbf{lower density} of a discrete set $\Lambda\subset \C$ is defined by
\[
 \liminf\limits_{r\rightarrow \infty}\frac{\#\{\lambda\in\Lambda:|\lambda|\leq r\}}{r}=:\ldens(\Lambda).
\]
In the early stages of the research on these operators, it was shown that a certain size of the unimodular point spectrum is sufficient for frequent hypercyclicity cf. \cite{frequentBayartGrivaux} and \cite{bayMathBook}. More precisely, in \cite{frequentBayartGrivaux} it is shown that a bounded operator on a Hilbert space is frequently hypercyclic if it has perfectly spanning set of eigenvectors associated to unimodular eigenvalues. A bounded operator $T$ on a Banach space $X$ is said to have \textbf{perfectly spanning set of eigenvectors associated to unimodular eigenvalues} if there exists a continuous probability measure $\mu$ on $\T:=\{z:|z|=1\}$ so that for every measurable set $A\subset \T$ with $\mu(A)=1$, the linear hull of  $\{\kernel(T-\lambda \I):\lambda\in A\}$ is dense in $X$, where $\I$ is the identity operator. This result was recently extended by S. Grivaux \cite{grivauxFrequentClass} to Banach spaces, i.e. her result states that every bounded operator on a Banach space that has perfectly spanning set of eigenvectors associated to unimodular eigenvalues is frequently hypercyclic. Conversely, it is shown by S. Shkarin in \cite{shkarinSpectrum}, that the spectrum of a frequently hypercyclic Banach space operator cannot have isolated points. 

For a bounded Banach space operator $T$, we denote by $\sigma(T)$ the spectrum of $T$. 
We extend the result of S. Shkarin to further necessary conditions on the spectrum of frequently hypercyclic operators.  \\
We apply the properties of numbers $\alpha_1,...,\alpha_m\in\R$ that are linearly independent over the field of rationals. That is,  $c_1\alpha_1+c_2\alpha_2+...+c_m\alpha_m=0$ with $c_1,...,c_m \in \Z$, implies that $c_1 =c_2=...=c_m=0$. 

\begin{unntheorem}[Kronecker's theorem]
Let $\alpha_1,...,\alpha_m\in\R$ be linearly independent over the field of rationals. Then for every choice of $\beta_1,...,\beta_m\in\R$ and $\varepsilon>0$ there exists a real $t$ such that $\sup_{j=1,...,m}|e^{ i\alpha_j t}-e^{i\beta_j}|<\varepsilon$.
\end{unntheorem}

Kronecker's theorem is an essential tool in the proof of our main result which states as follows.

\begin{theorem}\label{thendl}
Let $T$ be a bounded linear operator on a complex Banach space $X$ and $K$ a component of $\sigma(T)$ which is closed and open in $\sigma(T)$ and so that $K\subset \{z:|z|\leq 1\}$ and $A=K\cap \T$ is a finite set $A=\{e^{i\alpha_1},...,e^{i \alpha_m}\}$ such that $\alpha_1,...,\alpha_m$ are linearly independent over the field of rationals. Then $T$ is not frequently hypercyclic.
\end{theorem}

By means of some deep results on frequent hypercyclicity, we will obtain the following consequences of the latter theorem.

\begin{corollary}\label{fourNecCond}
Let $T$ be a bounded linear operator on a Banach space $X$ and $K$ a component of $\sigma(T)$ which is closed and open in $\sigma(T)$ and so that $K\subset \{z:|z|\leq 1\}$ and $A=K\cap \T$ is a finite set $A=\{e^{i\alpha_1},...,e^{i \alpha_m}\}$. 
\begin{enumerate}
\item[(1)]If there are some $n\in\N$ and $r\in \R $ so that $\alpha_1 n+r,...,\alpha_m n+r$ are linearly independent over the field of rationals, then $T$ is not frequently hypercyclic.
\item[(2)]If for every choice of integers $c_1, c_2,..., c_m$ with $\sum_{j=1}^m c_j=0$ and $\sum_{j=1}^m c_j \alpha_j =0$, it follows that $c_1=c_2=...=c_m=0$, then $T$ is not frequently hypercyclic. 
\item[(3)]If $A$ contains at most two elements, then $T$ is not frequently hypercyclic.
\item[(4)]If $\alpha_1/\pi,...,\alpha_m/\pi$ are rationals, then $T$ is not frequently hypercyclic.
\end{enumerate}

\end{corollary} 

It should be noted that \cite[Theorem 9.43]{linearChaosbookGEPer}, which is cited from a first version of \cite{grivauxFrequentClass}, contradicts our main result. However, the author of \cite{grivauxFrequentClass} has withdrawn this result and it does not appear in the final version of the latter work.  

Inspired by the work of S. Shkarin \cite{shkarinSpectrum}, the main result is shown by means of properties of entire functions at the end of this work. For that purpose we link the dynamics of the iterates $T^n$ with the translations $f(\cdot+n)$, $n\in \N$, of an entire function $f$. This is done in two steps, each of which establishes the (quasi) conjugacy of two operators. The following section deals with the first quasi conjugacy.


\section{Quasi conjugacy of a shift operator for holomorphic functions to bounded Banach space operators}
Let $X$, $Y$ be two topological vector spaces and $T:X\rightarrow X$, $S:Y\rightarrow Y$ continuous operators. Then $S$ is said to be \textbf{quasi conjugate} to $T$ if there exists a continuous mapping $\Psi :X\rightarrow Y$ with dense range and such that $\Psi\circ T=S\circ \Psi$. In case that $\Psi$ is an homeomorphism, that is $T$ is also quasi conjugate to $S$, $S$ and $T$ are called \textbf{conjugate}. (Quasi) conjugacy is a concept that is also used in the more general framework of dynamical systems and it is straightforward to show that the (frequent) hypercyclicity of $T$ implies the (frequent) hypercylicity of $S$ in the above situation, cf. \cite{linearChaosbookGEPer}. The aim of this section is to show the quasi conjugacy of a shift operator for holomorphic functions to bounded Banach space operators under certain conditions on their spectrum.  

We start with some notations.
For a compact set $K\subset \C$ we write $\Omega(K)$ for the unbounded component of $\C\setminus K$
and for an open set $\Omega\subset\C$, we denote by $H_0(\Omega)$ the space of holomorphic functions $f$ on $\Omega$ such that $|f(z)|\rightarrow 0$ as $|z|\rightarrow \infty$ in case that $\{z:|z|>R\}\subset \Omega $ for some $R>0$. Endowed with the usual topology of uniform convergence on compact subsets, it is known that $H_0(\Omega)$ is a Fr\'{e}chet space. 
We use the common notation $H(\C)$ for the space of entire functions endowed with the topology of uniform convergence on compact subsets.\\
The following terminology of cycles, introduced in \cite[Section 10.34]{rudinRealComplex}, will be convenient. In what follows a cycle will always be denoted by $\Gamma$, and by $|\Gamma|$ we mean the trace of $\Gamma$. 
We further say that $\Gamma$ is a\textbf{ Cauchy cycle} for a compact set $K$ in a domain $\Omega$, where $K\subset \Omega$, if $|\Gamma|\subset \Omega\setminus K$ and $\ind_\Gamma(u)=1$ for every $u\in K$ and $\ind_\Gamma(w)=0$ for every $w\in \C\setminus \Omega$. As usual, $\ind_\Gamma(u):=\frac{1}{2\,\pi\,i}\int_\Gamma \frac{1}{\xi-u}\,d\xi$, $u\in \C\setminus  |\Gamma|$.\\ 

For a bounded operator $T$ on a Banach space $X$ and a continuous functional $\Lambda$ on $X$ and a vector $x\in X$, consider the holomorphic function 
\[
z\mapsto \Lambda((z\I-T)^{-1}x)\ \ \ z\in \Omega(\sigma(T))
\]
which is induced by the \textbf{resolvent mapping}.
This mapping should be denoted by $\Psi_{\Lambda,T}(x)$. One observes that $\Psi_{\Lambda,T}(x)\in H_0(\Omega(\sigma(T)))$ and that $\Psi_{\Lambda,T}:X\rightarrow H_0(\Omega(\sigma(T)))$, $x\mapsto \Psi_{\Lambda,T}(x)$ defines a linear and continuous mapping (cf. \cite[Theorem 18.5]{rudinRealComplex}). 

For every $x$ and $\Lambda$, the function $\Psi_{\Lambda,T}(x)$ is known to admit the power series representation 
\begin{align}\label{powerSeriersRep}
\sum_{n=0}^\infty \Lambda(T^nx) z^{-(n+1)}
\end{align}
for $|z|>\max\{|u|:u\in \sigma(T)\}$ (cf. \cite{katznelson} page 245).

The power series in (\ref{powerSeriersRep}) gives rise to use $\Psi_{T,\Lambda}$ for the quasi conjugacy of the following shift operator to a bounded Banach space operator. For a domain $\Omega \subset \C $ such that $\{z:|z|>R\} \subset \Omega$ for some $R>0$, let $B: H_0(\Omega)\rightarrow H_0(\Omega)$ be the linear operator defined by $BF(z):=zF(z)-F^{'}(\infty)$. This operator coincides with the shift   
\[
\sum\limits_{n=0}^\infty \frac{a_n}{z^{n+1} }\mapsto \sum\limits_{n=0}^\infty \frac{a_{n+1}}{z^{n+1}}.
\]
in the power series representation.
For $\Omega=\Omega(\sigma(T))$, one immediately observes that $\Psi_{\Lambda,T} T =B\Psi_{\Lambda,T}$ so that in case that $\Psi_{\Lambda,T}$ has dense range, we obtain that $B$ is quasi conjugate to $T$. Our next result gives a sufficient condition for this situation.\\

\begin{theorem}\label{dense}
Let $X$ be a Banach space, $T$ a linear operator on $X$ not identical zero and such that $\sigma(T)$ has no isolated points. Then there exists a bounded  functional $\Lambda$ on $X$ such that $\Psi_{\Lambda,T}:X\rightarrow H_0(\Omega(\sigma(T)))$ has dense range.
\end{theorem}

\begin{corollary}\label{quasic}
Let $X$ be a Banach space, $T$ a linear operator on $X$ not identical zero and such that $\sigma(T)$ has no isolated points. Then there exists a bounded functional $\Lambda$ on $X$ such that $B$ is quasi conjugate to $T$ via $\Psi_{\Lambda,T}$.
\end{corollary}

In the proof of Theorem \ref{dense} we will apply the following result due to K. Ball (cf. \cite{ball_eins}). 
\begin{unntheorem}[Ball]\label{funcEx}
Let $X$ be a Banach space, $(x_n)$ a sequence in $X$ such that $||x_n||=1$ for all $n\in \N$, and $(a_n)$ a sequence of non-negative numbers such that $\sum_{n=1}^\infty a_n\leq1$. Then there exists a continuous functional $\Lambda$ on $X$ such that $|\Lambda(x_n)|\geq a_n$ for all $n\in\N$ and $||\Lambda||\leq 1$. 
\end{unntheorem}

\begin{proof}[Proof of Theorem \ref{dense}]
We construct a functional $\Lambda$ such that the linear hull of the images under $\Psi:=\Psi_{\Lambda,T}$ of vectors corresponding to the approximate spectrum of $T$ is dense in $H_0(\Omega(\sigma(T)))$.
Every point on the boundary $\partial\sigma(T)$ of $\sigma(T)$ is part of the approximate spectrum (cf. \cite[Proposition VII.6.7]{conwayFA}). That means, for every $\lambda\in\partial\sigma(T)$ and every $\varepsilon>0$ we can choose some $x\in X$, $||x||=1$, such that $||Tx-\lambda x||<\varepsilon$.
Let $(\lambda_n)$ be a countable dense subset of the approximate spectrum. We choose vectors $(x_n)$, $n\in\N$, corresponding to $(\lambda_n)$ in a way that $||x_n||=1$, $||Tx_n-\lambda_n x_n||<1/n^3$ for every $n\in\N$. Then the above theorem yields a bounded functional $\Lambda$ such that $|\Lambda(x_n)|\geq 1/n^2$.
Now let $L$ be an arbitrary compact set in $\Omega(\sigma(T))$. The continuity of $\Psi$ implies the existence of some $M>0$ such that $\sup_{z\in L}|\Psi(y)(z)-\Psi(x)(z)|<M||y-x||$ for all $x,y\in X$. For an arbitrary $\lambda\in \partial \sigma(T)$ with $2 \sup_{z\in L}|1/(z-\lambda)|=C<\infty$ and $\varepsilon>0$, we may choose $\lambda_n$ being so close to $\lambda$ that 
\begin{align}\label{propOfn}
\sup_{z\in L}\left|\frac{1}{z-\lambda_n}\right|<C,\ \ M C \frac{1}{n} <\frac{\varepsilon}{2} \textnormal{  and  } \sup\limits_{z\in L}\left| \frac{1}{z-\lambda}-\frac{1}{z-\lambda_n}\right|<\frac{\varepsilon}{2}.
\end{align}
For the vector $x_n$ that corresponds to this $\lambda_n$, the property $B \Psi=\Psi T$ implies $\sup_{z\in L}|B\Psi(x_n)(z)-\lambda_n\Psi(x_n)(z)|<M/n^3$. With the representation $\Psi(x_n)(z)=\sum_{n=0}^\infty \frac{a_n}{z^{n+1} }$, this is 
\[
\sup\limits_{z\in L} \left|z \,\Psi(x_n)(z)-a_0-\lambda_n\Psi(x_n)(z)\right|<M \frac{1}{n^3},
\]
or equivalently
\[
\sup\limits_{z\in L} \left|z-\lambda_n\right|\,\left|\Psi(x_n)(z)-\frac{a_0}{z-\lambda_n}\right|<M \frac{1}{n^3}.
\]
The application of the first inequality of (\ref{propOfn}) to the above inequality gives 
\[
\sup\limits_{z\in L} \left|\Psi(x_n)(z)-\frac{a_0}{z-\lambda_n}\right|< \frac{M C}{n^3}.
\]
By construction, $|\Lambda(x_n)|=a_0\geq 1/n^2$, so that by means of the second and third inequality of (\ref{propOfn}) we can conclude
\[
\sup\limits_{z\in L} \left|\frac{\Psi(x_n)(z)}{a_0}-\frac{1}{z-\lambda}\right|< \frac{M C}{n}+\frac{\varepsilon}{2}<\varepsilon.
\] 
Since $\sigma(T)$ has no isolated points, $\linspan\{z\mapsto 1/(z-\lambda):\lambda\in \partial \sigma(T)\}$ is dense in $H_0(\Omega(\sigma(T)))$ by a variant of Runge's theorem (cf. \cite[Theorem 10.2]{lueckinRubel}). Thus the last inequality proves the assertion.
\end{proof}


\section{Connection of iterates of bounded operators to translations of entire functions}

In this section we link the spectrum of linear bounded operators with the growth of entire functions of exponential type. 
For the sake of completeness, we recall that an entire function $f$ is said to be \textbf{of exponential type $\tau$} if 
\[
 \limsup\limits_{r\rightarrow\infty}\frac{\log M_f(r)}{r}=:\tau(f)=\tau
\]
and, more general, $f$ is said to be a function of exponential type when the above $\limsup$ is finite.
For a function $f$ of exponential type, $\Bo f(z):=\sum_{n=0}^\infty f^{(n)}(0)/z^{n+1}$ is called the \textbf{Borel transform} of $f$. Obviously, the function $\Bo f$ is a holomorphic on some neighbourhood of infinity that vanishes at infinity. \\ 
Our proofs are based on the investigation of the possible behaviour of functions of exponential type in the direction of the positive real line under some conditions on a domain of holomorphy of their Borel transform. For a given entire function of exponential type $f$ we denote by $\A(f)$ the set of compact sets $K\subset\C$ such that the Borel transform $\Bo f$ admits an analytic continuation to $\Omega(K)$. We also use the notation $\Bo f$ for an analytic continuation of the Borel transform. The \textbf{conjugate indicator diagram} of $f$ is the intersection of all convex sets in $\A(f)$ and will be denoted by $K(f)$ in what follows. The convex hull of a set $M\subset\C $ is denoted by $\conv(M)$. We further remark that $K(f)\subset\{z:|z|\leq r\}$, $r\geq 0$, implies that $f$ is a function of exponential type less or equal than $r$. The most natural examples of entire functions of exponential type are the \textbf{exponential functions} $z\mapsto e^{\alpha z}$, $\alpha\in \C$, which are abreviated by $e_\alpha$ in the following. For these functions we have $\{\alpha\}=K(e_\alpha)\in \A(e_\alpha)$ and $\Bo e_\alpha(z)=1/(\alpha-z)$.

For a given compact and convex set $K\subset \C$ we denote by $\Exp(K)$ the set of entire functions $f$ of exponential type that satisfy $K(f)\subset K$. Actually, this set is a vector space that can be endowed with a natural topology which makes it a
Fr\'{e}chet space. To see this, consider the \textbf{support function} defined by 
\[
H_K(z) :=\sup\limits\{\Reee(zu):u\in K\},  \quad z \in \C
\]
of the convex set $K$. The support function of the conjugate indicator diagram gives an estimate of the growth of its entire function in the following manner. If $f$ is an entire function of exponential type then 
the \textbf{indicator function} $h_f$, defined by
\[
 h_f(\theta):=\limsup\limits_{r\rightarrow\infty} \frac{\log |f(re^{i\theta})|}{r},\ \ \theta\in[-\pi,\pi]
\]
satisfies
\begin{equation}\label{suppIndic}
 r\, h_f(\theta)=H_{K(f)}(z)
\end{equation}
for $z=re^{i\theta}$ (cf. \cite{berenstein}). The topology on $\Exp(K)$ is induced by the norms
\[
||f||_{K,n}:=\sup\limits_{z\in\C}|f(z)|\,e^{-H_{K}(z)-\frac{1}{n}|z|}, \quad n\in\N.
\]
For more detailed information we refer to  \cite{martineau}, \cite{morimoto}, \cite{berenstein}.

We establish a connection between the iterates of bounded operators and the translations of entire functions of exponential type. For that purpose we introduce a transform that is studied in more details in \cite{beiseMuellerDifferential}. A function $f$ of exponential type admits the representation
\[
f(z)=\frac{1}{2\pi i} \int_\Gamma \Bo f(\xi)\, e^{\xi z}\, d\xi
\]
where $\Gamma$ is a Cauchy cycle for $K(f)$ in $\C$. This integral formula is known as the \textbf{P\'{o}lya representation}. Given a function $\varphi$ holomorphic on some open neighbourhood of $K(f)$, we set 
\[
\Phi_\varphi f(z)=\frac{1}{2\pi i} \int_\Gamma \Bo f(\xi)\, e^{\varphi(\xi) z}\, d\xi
\]
where $\Gamma$ is a Cauchy cycle for $K(f)$ in the domain of holomorphy of $\varphi$. One directly observes that $\Phi_\varphi f$ is a function of exponential type. Actually, it is not hard to see that in the last two integral formulas, $\Gamma$ can be chosen as a Cauchy cycle for any $K\in \A(f)$ and so that $|\Gamma|\subset \Omega(K)$, provided that $\varphi$ is holomorphic on some open neighbourhood of $K$ in the last formula. Thus for all $K\in \A(f)$ and $\varphi$ holomorphic on some neighbourhood of $K$, the transform $f\mapsto \Phi_\varphi f$ is well defined.

\begin{proposition}\label{phicontinuation}
Let $f$ be a function of exponential type, $K\in \A(f)$ and $\varphi$ a holomorphic function on some open neighbourhood of $K$, then $\varphi(K)\in \A(\Phi_\varphi f)$. 
\end{proposition}
\begin{proof}
For an open neighbourhood $U$ of $K$, such that $\varphi$ is holomorphic on $U$, let $\Gamma$ be a Cauchy cycle for $K$ in $U$ with $|\Gamma|\subset \Omega(K)$. Then for $w\in\C\setminus \varphi(K)$ such that $\{z:\varphi(z)=w\}\cap U=\emptyset$ we set
\[
H_\varphi(w):= \frac{1}{2\pi i} \int_{\Gamma} \frac{\Bo f(\xi)}{w-\varphi(\xi)}\,d\xi.
\]
As $U$ can be arbitrary close to $K$, we obtain a holomorphic function $H_\varphi\in H_0(\Omega(\varphi(K)))$.
Now, given a neighbourhood $U$ of $K$ as above, and $\Gamma$ a Cauchy cycle for $K$ in $U$, we take a Cauchy cycle $\Gamma_0$  for $\varphi(K)\cup \varphi(|\Gamma|)$ in $\varphi(U)$. Then 
\begin{align*}
\frac{1}{2\pi i} \int_{\Gamma_{0}} H_\varphi(w)  \,e^{w z}\,dw &=\frac{1}{2\pi i} \int_{\Gamma}\Bo f(\xi)\,\frac{1}{2\pi i} \int_{\Gamma_{0}}\frac{e^{w z}}{w-\varphi(\xi)}\, dw \,d\xi\\[2mm]
&=\frac{1}{2\pi i} \int_{\Gamma}\Bo f(\xi)\,e^{\varphi(\xi)z}\,d\xi\\[2mm]
&=\Phi_\varphi f(z)
\end{align*} 
by the Cauchy integral formula. This proves $H_\varphi=\Bo \Phi_\varphi f$ on $\Omega(\varphi(K))$. 
\end{proof}

Let us denote by $T_1$ the translation operator that maps an entire function $f$ to $f(\cdot +1)$ and by $D$ the differentiation operator.
\begin{proposition}\label{interpolation}
Let $f$ be a function of exponential type and $K\in \A(f)$ so that a branch of the logarithm $\log$ is defined on some open neighbourhood of 
$K$, then we have $\Phi_{\log} D f=T_1 \Phi_{\log} f$.
\end{proposition}

The assertion of Proposition \ref{interpolation} can be immediately deduced from the following identity
\[
\frac{1}{2\pi i} \int_{\Gamma} \Bo( Df(\xi)) \xi^z d\xi=\frac{1}{2\pi i} \int_{\Gamma} \left(\xi\Bo f(\xi)- f(0)\right)\xi^z d\xi= \frac{1}{2\pi i} \int_{\Gamma} \Bo f(\xi)\xi^{z+1} d\xi,
\]
where $\xi^z=e^{z \log \xi}$. For more details we refer to proof of \cite[Proposition 3.4 (2)]{beiseMuellerDifferential}.

Setting $F(\xi):=1/\xi\, \Bo f(1/\xi)$ in the situation of Proposition \ref{interpolation}, $F$ is holomorphic off $K^{-1}$ and vanishes at infinity. The latter result then states that the Taylor coefficients of $F$ at the origin coincide with evaluations of $\Phi_{\log} f$ at the non-negative integers, that is $F^{(n)}(0)/n!=\Phi_{\log} f(n)$ for $n=0,1,...$ . In these terms the assertion of the latter proposition is well known from the theory of analytic continuation, c.f.  \cite[Theorem 1.3.III]{bieberbach}.  

\begin{corollary}\label{shiftalpha}
Let $f$ be a function of exponential type and $\alpha\in \C$, then $ \alpha+K\in \A(e_\alpha f)$ for all $K\in \A(f)$.
\end{corollary}

\begin{proof}
Making use of the P\'{o}lya representation, we have
\[
(e_\alpha f)(z)=\frac{1}{2\pi i} \int_\Gamma \Bo f(\xi)\, e^{\alpha z+\xi z}\, d\xi.
\]
With $\varphi(\xi)=\alpha +\xi$, this is $e_\alpha f=\Phi_\varphi f$ so that the assertion follows from Proposition \ref{phicontinuation}.
\end{proof}

\begin{lemma}\label{decomp}
Let $f$ be an entire function of exponential type and $K\in \A(f)$. Then for every finite collection of compact sets $K_0,...,K_m$ so that $\bigcup_{j=0}^m K_j=K$, there exist functions $f_0,...,f_m$ of exponential type such that $\sum_{j=0}^m f_j=f$ and $K_j\in \A(f_j)$ for all $j=0,...,m$. 
\end{lemma}

\begin{proof}
Due to a classical result of N. Aronszajn \cite{aronszajn} (see also \cite{hoermandernComp} or \cite{muellerWengenAronTheorem} for an alternative proof), the Borel transform of $f$ can be written as $\Bo f=\sum_{j=0}^m F_j$ with $F_j\in H_0(\C\setminus K_j)$, $j=0,...,m$. The P\'{o}lya representation suggests to define 
\[
f_j(z):=\frac{1}{2\pi i} \int_{\Gamma_j} F_j(\xi)\, e^{\xi z}\, d\xi
\]
with $\Gamma_j$ being a Cauchy cycle for $K_j$ in $\C$ for $j=0,...,m$. Obviously, this definition is independent of the particular choice of the $\Gamma_j$ and each $f_j$ is a function of exponential type. The functions $F_j$, $j=0,...,m$, admit a power series representation at infinity, say $\sum_{\nu=0}^\infty F_{\nu,j}/z^{(n+1)}$, for $|z|$ sufficiently large. One directly verifies 
\[
F_{\nu,j}=\frac{1}{2\pi i} \int_\Gamma F_j(\xi)\xi^\nu d\xi=f_j^{(\nu)}(0) \textnormal{  for all }\ j=0,...,m,\ \nu=0,1...\,
\]
so that $F_j=\Bo f_j$ on the unbounded component of $K_j$.
\end{proof}

Our next result gives the connection between (frequently) hypercyclic vectors for a bounded operator $T$ and (frequently) hypercyclic entire funtions $f$ with respect to the translation operator $T_1$.

\begin{proposition}\label{banachToShift}
Let $X$ be a complex Banach space, $T$ a bounded operator on $X$ such that $\sigma(T)$ has no isolated points and such that for some $r\in (0,1)$ a branch of the logarithm $\log$ is defined on an open neighbourhood of $\sigma(T)\setminus\{z:|z|<r\}$. Then, if $T$ is (frequently) hypercyclic on $X$, there exists an entire function $f$ of exponential type that is (frequently) hypercyclic for the translation operator $T_1$ with respect to the topology of $H(\C)$ such that $\log(\sigma(T)\setminus\{z:|z|<r\})\in \A(f)$. 
\end{proposition} 

In the following proof it is necessary to consider quasi conjugacy for general dynamical systems as it is introduced in \cite[Definition 1.5]{linearChaosbookGEPer}. Let $S_1:M_1 \rightarrow M_1, $ and $S_2:M_2\rightarrow M_2$ be continuous mappings on metric spaces $M_1,M_2$. Then, $S_2$ is said to be quasi-conjugate to $S_1$ if there exists a continuous mapping $R:M_1\rightarrow M_2$ with dense range such that $R\circ S_1=S_2\circ R$. If for some $x\in M_1$, the orbit $\{S_1^nx:n\in\N\}$ is dense in $M_1$, then so is the orbit $\{S_1^ny:n\in\N\}$ of $y=R(x)$ in $M_2$ (cf. \cite[Proposition 1.19]{linearChaosbookGEPer}). It is also obvious that if $\{S_1^nx:n\in\N\}$ is frequently dense in the sense that for every non-empty open set $U$ of $M_1$ the sequence $\{n: S_1^nx\in U\}$ has positive lower density, the analog property holds for $Rx=y$. Thus, if $M_2$ is a topological vector space and $S_2$ is a continuous operator, $y$ is a frequently hypercyclic vector for $S_2$ in the latter situation. 

\begin{proof}

We assume that $T$ is (frequently) hypercyclic. Then necessarily, $\sigma(T)\setminus\{z:|z|<r\}\neq\emptyset$ since otherwise $||T^n||\rightarrow 0$ as $n\rightarrow \infty$. 
Corollary \ref{quasic} implies the existence of a functional $\Lambda$ such that $\Psi_{\Lambda,T} x=G$ is (frequently) hypercyclic for the shift operator $B$ on $H_0(\Omega(\sigma(T))$. The Borel transform $\Bo=\Bo_K: \Exp(K)\rightarrow H_0(\C\setminus K)$ is an isomorphism for every compact convex set $K$ (cf. first chapter of \cite{morimoto}). We choose $K$ as a convex set so that $G$ is holomorphic off $K$. Then $G|_{\Omega(K)}=G$ is still (frequently) hypercyclic on $H_0(\Omega(K))$, and setting $\Bo_K^{-1} G= g$, we have $\sigma(T)\in \A(g)$. Further, one verifies that  $\Bo_{K} D =B \Bo_{K}  $. Thus $B$ is conjugate to $D$ via $\Bo_{K}$ and hence $g$ is (frequently) hypercyclic for $D$ with respect to the topology of $\Exp(K)$. 
By means of Lemma \ref{decomp} we can decompose $g=g_0+g_1$ such that $K_0:=\{z:|z|\leq r\}\in \A(g_0)$ and $K_1=\sigma(T)\setminus\{z:|z|<r\}\in \A(g_1)$. In terms of the P\'{o}lya representation 
\[
D^n g_0(z)=\frac{1}{2\pi i} \int_\Gamma \Bo g_0(\xi)\,\xi^n\, e^{\xi z}\, d\xi,
\]
where $\Gamma$ can be chosen such that $|\Gamma|\subset \{z:|z|<1\}$. An elementary estimation of the above integral yields $||D^n g_0||_{K,m} \rightarrow 0$ for every $m\in\N$ as $n$ tends to infinity. Thus it turns out that $g_1$ is still (frequently) hypercyclic for $D$ on $\Exp(K)$. \\
We define $M$ as the set of all $h\in \Exp(K)$ so that $K_1\in \A(h)$. Considering the conjugacy of $B$ and $D$, one observes that $M$ is invariant under $D$. Hence, $M$ endowed with a metric coming from $\Exp(K)$ and $D$ form a dynamical system in the sense of \cite[Definition 1.1]{linearChaosbookGEPer}. 
By the Cauchy integral formula and the observation that $\Bo_K e_{\alpha}=\xi\mapsto 1/(\xi-\alpha)$, one deduces $\Phi_{\log} e_\alpha = e_{\log(\alpha)}$ for every $\alpha \in K_1$. Further, since $\sigma(T)$ is assumed to have no isolated points, $K_1$ is an infinite set. By means of \cite[Proposition 2.5 (2)]{beiseMuellerDifferential} this implies that $\Phi_{\log}|_M\rightarrow \Exp(\conv(\log(K_1)))$ has dense range. \\
Finally, with the remark preceding this proof and Proposition \ref{interpolation}, $\Phi_{\log} g_1=:f$ is (frequently) hypercyclic on $\Exp(\conv(\log(K_1)))$, and Proposition \ref{phicontinuation} implies $\log(\sigma(T)\setminus\{z:|z|<r\})\in \A(f)$. The (frequent) hypercyclicity with respect to the topology of $H(\C)$ follows since the latter is weaker than the topology on $\Exp(\conv(\log(K_1)))$.
\end{proof}

\begin{theorem}\label{entireNotFrequently}
Let $f$ be an entire function of exponential type and $K\in \A(f)$ such that $K\subset \{z:\Reee(z)\leq 0\}$ and $K\cap \{z:\Reee(z)=0\}$ is a finite set $\{i\alpha_1,...,i\alpha_m\}$ such that $\alpha_1,...,\alpha_m$ are linearly independent over the field of rationals. Then $f$ is not frequently hypercyclic for the translation operator $T_1$ in the topology of $H(\C)$.
\end{theorem}

\begin{lemma}\label{kroneckerLemma}
Let $\alpha_1,...,\alpha_m \in\R$ be linearly independent over the field of rationals and $\varepsilon>0$. Then there exists a $t_0:=t_{0,\varepsilon} >0$ such that for every choice of $\beta_1,...,\beta_m\in [0,2\pi]$ and $x\in\R$, there exists a $t\in [x,x+t_0]$ that satisfies $\sup_{j=1,...,m}|e^{ i\alpha_j t}-e^{i\beta_j}|<\varepsilon$.
\end{lemma}

\begin{proof}
For a given $b=(\beta_1,...,\beta_m)\in [0,2\pi]^m$, Kronecker's theorem ensures the existence of a $t_b\in \R$ such that $\sup_{j=1,...,m}|e^{ i\alpha_j t_b}-e^{i\beta_j}|<\varepsilon$. By continuity, it follows that there is a $\delta_b>0$ so that still $\sup_{j=1,...,m}|e^{ i\alpha_j t_b}-e^{i\rho_j}|<\varepsilon$ for every $(\rho_1,...,\rho_m)$ in $U_{\delta_b}(b)=\{r\in [0,2\pi]^m:||r-b||_2<\delta_b\}$. Due to the compactness of $[0,2\pi]^m$, there are finitely many $b_1,...,b_l\in [0,2\pi]^m$ so that $[0,2\pi]^m\subset \bigcup_{\nu=1}^l U_{\delta_{b_\nu}}$ and thus $t_0:=2\max\{|t_{b_1}|,...,|t_{b_l}|\}$ has the claimed property for $x=-t_0/2$. For general $x\in \R$ the above yields a $t\in [-t_0/2, t_0/2]$ such that $\sup_{j=1,...,m}|e^{ i\alpha_j t}-e^{i\beta_j-i\alpha_j(x+t_0/2)}|<\varepsilon$ so that $\tilde{t}:=x+t_0/2+t\in[x,x+t_0]$ has the claimed property.

\end{proof}

\begin{proof}[Proof of Theorem \ref{entireNotFrequently}]
By means of Lemma \ref{kroneckerLemma}, there is a $t_0>0$ such that for any $\beta_1,...,\beta_m\in [0,2\pi]$ and $x\in [0,\infty)$, there exists a $t\in [x,x+t_0]$ that satisfies 
\begin{equation}\label{kronEqua}
\sup_{j=1,...,m}|e^{ i\alpha_j t}-e^{i\beta_j}|<\frac{1}{2}.
\end{equation}
For the purpose of showing the result by contradiction, we assume that $f$ is frequently hypercyclic. This implies the existence of an increasing sequence $N\in \N^{\N}$ so that $\ldens(N)=d>0$ and with 
\begin{equation}\label{freqCond}
\sup_{x\in [n,n+t_0]} |f(x)-1|<\frac{1}{2} \mbox{ for every }n\in N.
\end{equation}
We thin out the sequence $N$ to a sequence $N_1$ in the following manner: Set $n_1:=\min\{n:n\in N\}$ and $n_j:=\min\{n\in N :n_{j-1}+t_0<n\}$ for $j>1$. One immediately deduces that  $\ldens(N_1)\geq d/t_0$. 
Considering the geometric assumption of $K$ and applying Lemma \ref{decomp}, we may decompose $f$ into a finite sum of functions of exponential type $\sum_{j=0}^m f_j=f$ with $K_j\in \A(f_j)$, $j=0,...,m$ so that 
\begin{equation}\label{decom_f}
K_j\subset \{z:|z-i\alpha_j|<d/2t_0m\} \mbox{ for } j=1,...,m
\end{equation}
and $K_0\subset \{z:\Reee(z)<0\}$. The latter condition implies that $h_{f_0}(0)<0$, which can be verified by $(\ref{suppIndic})$, and that gives $|f_0(x)|\rightarrow 0$ for $x\in\R$ as $x\rightarrow +\infty$. Hence, for sufficiently large integers in $N_1$, (\ref{freqCond}) still holds for $\sum_{j=1}^m f_j$, i.e.    
\begin{equation}\label{freqCondTwo}
\sup_{x\in [n,n+t_0]} |\sum_{j=1}^m f_j(x)-1|<\frac{1}{2} \mbox{ for every }n\in N_2,
\end{equation}
where $N_2:=\{n\in N_1:n>n_0\in \N \}$ for some $n_0 \in\N$. It is clear that $\ldens(N_2)\geq d/t_0$. Due to Corollary \ref{shiftalpha}, the functions $h_j =e_{-i\alpha_j}f_j$, $j=1,...,m$ are entire functions of exponential type with $L_j=K_j-i\alpha_j\in \A(h_j)$ and by (\ref{decom_f}), $L_j\subset \{z:|z|<d/2t_0m\}$. The last inclusion implies that each $h_j$ is a function of exponential type less than $d/2t_0m$.\\
We show that at least one of $\Reee( h_j)$ or $\Imm(h_j)$, $j=1,...,m$ has a zero on $[n,n+t_0]$ for every $n\in N_2$. For this, let us fix an $n\in N_2$ and suppose that the following holds:
\begin{enumerate}\label{assumpt}
\item[(A)]For all $j=1,...,m$, we have $\Reee(h_j(n))\neq 0$, $\Imm(h_j(n))\neq 0$ and  neither $\Reee(h_j)$ nor $\Imm(h_j)$ have a change of sign on $[n,n+t_0]$.
\end{enumerate}
We choose $b=(\beta_1,....,\beta_m)\in [0,2\pi]^m$ in the following way
\begin{align}\label{betachoice}
&\textnormal{if } \Reee( h_j)(n)> 0 \textnormal{ and }  \Imm(h_j)(n)> 0 \textnormal{, then } \beta_j=3/4\,\pi \notag\\
&\textnormal{if } \Reee( h_j)(n)>0 \textnormal{ and }  \Imm(h_j)(n)< 0 \textnormal{, then } \beta_j=-3/4\,\pi \\
&\textnormal{if } \Reee( h_j)(n)< 0 \textnormal{ and }  \Imm(h_j)(n)> 0 \textnormal{, then } \beta_j=1/4\,\pi \notag\\
&\textnormal{if } \Reee( h_j)(n)< 0 \textnormal{ and }  \Imm(h_j)(n)< 0 \textnormal{, then } \beta_j=-1/4\,\pi \notag.
\end{align}
Then, considering the arcs
\begin{align*}
&S_1:=\{e^{i \lambda}\in \T: 0< \lambda <  \pi/2\},  \\ 
&S_2:=\{e^{i \lambda}\in \T: \,  \pi/2 < \lambda <\pi\}, \\ 
&S_3:=\{e^{i \lambda}\in \T: \, -1/2\,\pi <\lambda < 0\}, \\
&S_4:=\{e^{i \lambda}\in \T: -\pi < \lambda<  -1/2\, \pi\},
\end{align*}
one verifies that under assumption (A), $h_j(t)/|h_j(t)|$ stays in $S_1,S_2,S_3,S_4$ for all $t\in [n,n+t_0]$ if $h_j(n)/|h_j(n)|$ is already contained in the respective arc. Using the choice of $b$ according to (\ref{betachoice}), this implies that for all $\theta\in \R$ so that $|e^{i\theta}-e^{i\beta_j}|<1/2$, the rotation by $\theta$ is such that $e^{i\theta} h_j(t)$ is contained in the left half plane.  \\
Considering (\ref{kronEqua}), we can conclude that there exists a $t\in [n,n+t_0]$ such that $e^{i\alpha_j t} h_j(t)\in\{z:\Reee(z)<0\}$ for $j=1,...,m$. However, this contradicts (\ref{freqCondTwo}), so that assumption (A) must be false.
By the intermediate value theorem, we obtain our assertion that one of $\Reee( h_j)$ or $\Imm(h_j)$, $j=1,...,m$, has a zero on $[n,n+t_0]$.  \\
Setting $R_j(z):=\sum_{\nu=0}^\infty \Reee(h_j^{(\nu)}(0))/\nu! \,z^\nu $ and $I_j(z):=\sum_{\nu=0}^\infty \Imm(h_j^{(\nu)}(0))/\nu! \,z^\nu $ we obtain functions of exponential type less than $d/2t_0m$. This is clear since $\tau(h_j)\leq d/2 t_0 m$ and $|h_j^{(\nu)}(0)|\geq |\Reee(h_j^{(\nu)}(0))|$, $|h_j^{(\nu)}(0)|\geq |\Imm(h_j^{(\nu)}(0))|$. Further, taking into account that $R_j(x)=\Reee(h_j(x))$ and $I_j(x)=\Imm(h_j(x))$ for all real $x$, we obtain that the product $P:=\prod_{j=1}^m R_j\,I_j$ has a zero $[n,n+t_0]$ for every $n\in N_2$. As a product of $2m$ functions of exponential type less than $d/2t_0 m$, $P$ is a function of exponential type less than $d/t_0$ (cf. \cite{boas}). Since $\ldens(N_2)\geq d/t_0$, this can only happen if $P\equiv 0$ (cf. \cite[Theorem 2.5.13]{boas}). Consequently, $R_j$ or $I_j$ is constantly zero for at least one $j\in\{1,...,m\}$.\\
Without restriction, say $R_1\equiv 0$ or $I_1\equiv 0$. 
We then proceed in the same way as above without considering $\Reee(h_1)$ or $\Imm(h_1)$, respectively. In case that $R_1\equiv 0$ there is no need to impose a condition on $\beta_1$ in (\ref{betachoice}) since $h_1(t)\in i\R$ for every $t\in \R$ which is sufficient to obtain the impossibility of (\ref{freqCondTwo}). If $I_1\equiv 0$, we start at the beginning and consider $if$ instead of $f$.\\
Inductively, we finally obtain that $h_j\equiv 0$ for all $j=1,...,m$, and have a contradiction to our assumption that $f$ is frequently hypercyclic. 
\end{proof}

\begin{proof}[Proof of Theorem \ref{thendl}]We suppose that $T$ is frequently hypercyclic.
Then the spectrum of $T$ cannot have isolated points by \cite[Theorem 1.2]{shkarinSpectrum}. 
As in the proof of \cite[Theorem 1.2]{shkarinSpectrum} we may assume that $K=\sigma(T)$.
By the condition that the intersection of $K$ and $\{z:|z|=1\}$ is a finite set, there is an $0<r<1$ such that a branch of the logarithm $\log$ is defined on $\sigma(T)\setminus\{z:|z|<r\}$.
Thus Proposition \ref{banachToShift} yields the existence of a frequently hypercyclic $f$ with respect to the translation operator $T_1$ (in the topology of $H(\C)$) that is of exponential type and such that $ L=\log (\sigma(T)\setminus\{z:|z|<r\})\in A(f)$. But then $\{i\alpha_1,...,i\alpha_m\}=L\cap\{z:\Reee(z)=0\}$ so that we have a contradiction to Theorem \ref{entireNotFrequently}.
\end{proof}

\begin{proof}[Proof of Corollary \ref{fourNecCond}]
As in the proof of \cite[Theorem 1.2]{shkarinSpectrum} we may assume that $K=\sigma(T)$. 
In order to see (1), consider the operator $S=e^{i r}T^n$. Then $\sigma(S)\subset \{z:|z|\leq 1\}$ and $\sigma(S)\cap\{z:|z|=1\}=\{e^{ir+i\alpha_1 n},...,e^{ir+i\alpha_m n}\}$. Now, if $T$ is frequently hypercyclic, the application of \cite[Theorem 9.27, Theorem 9.35]{linearChaosbookGEPer} yields that $S$ is frequently hypercyclic. However, this contradicts Theorem \ref{thendl} since $r+\alpha_1 n,...,r+\alpha_m n$ are assumed to be linearly independent over the field of rationals.

For the proof of (2) we define $C$ as the set of all $c=(c_1,...,c_m)\in \Z^m$ so that $\sum_{j=1}^m c_j \neq 0$. For a given $r\in\R$ the assumption of (2) implies that $\alpha_1+r,...,\alpha_m+r$ are not linearly independent over the field of rationals if and only if $\sum_{j=1}^m c_j(\alpha_j+r) =0$ for some $c=(c_1,...,c_m) \in C$. The latter is equivalent to 
\[
r=-\frac{1 }{\sum\limits_{j=1}^m c_j} \,\sum\limits_{j=1}^m c_j\alpha_j \textnormal{  for some } c\in C.
\]
Since $C$ is countable, there exists an $r_0$ such that the above equality does not hold. Considering (1), $T$ cannot be frequently hypercyclic. 

In the situation of (3) one verifies that the condition of (2) holds in case that $A$ contains two points. If $A$ is a singleton, it is immediately clear that $A+r$ is linearly independent over the fields of rational for some $r\in\R$. Thus, in both cases the assertion follows from (1) and (2).

Under the conditions of (4), there is a positive integer $k$ such that $\sigma(T^k)\cap \{z:|z|=1\}=\{1\}$. With (3) we can conclude that $T^k$ is not frequently hypercyclic and considering \cite[Theorem 9.27]{linearChaosbookGEPer}, this shows that also $T$ cannot be frequently hypercyclic.
\end{proof}

In view of Theorem \ref{thendl} and Corollary \ref{fourNecCond} it is natural to ask the following
\begin{question}
Let $T$ be a frequently hypercyclic bounded operator on a Banach space such that $\sigma(T)\subset\{z:|z|\leq 1\}$. 
\begin{enumerate}
\item[(a)] Is it possible that $|\sigma(T)\cap \T|=3$?
\item[(b)] More general, is it possible that $|\sigma(T)\cap \T|<\infty$?
\end{enumerate}
\end{question}

\subsection*{Acknowledgements}
I thank Karl-Goswin Grosse-Erdmann and Jürgen Müller for very interesting and helpful discussions.

\bibliographystyle{plain}                           
\bibliography{referencesList}

\end{document}